\documentclass[leqno]{amsart}

\usepackage{amsmath,amssymb,amsthm}

\usepackage{chngpage,graphicx}

\usepackage{comment}

\numberwithin{equation}{section}

\newcommand{\bZ}{\mathbb{Z}}
\newcommand{\bR}{\mathbb{R}}
\newcommand{\bfs}{\mathbf{s}}
\newcommand{\bft}{\mathbf{t}}
\newcommand{\cC}{\mathcal{C}}
\newcommand{\el}[1]{\left| #1 \right|}
\newcommand{\vol}[2]{\mathrm{Vol}_{#1}(#2)}
\newcommand{\Leg}[2]{\left(\frac{#1}{#2}\right)}

\newtheorem{theorem}{Theorem}[section]
\newtheorem{corollary}[theorem]{Corollary}
\newtheorem{lemma}[theorem]{Lemma}
\newtheorem{open}{Open Problem}
\newtheorem{proposition}[theorem]{Proposition}

\begin{document}
\title{Constructing $k$-radius sequences}
\author{Simon R. Blackburn and James F. McKee}
\address{Department of Mathematics\\
Royal Holloway, University of London\\
Egham, Surrey TW20 0EX\\
United Kingdom}
\email{s.blackburn@rhul.ac.uk, james.mckee@rhul.ac.uk}
\subjclass[2000]{94A55}

\begin{abstract}
An $n$-ary $k$-radius sequence is a finite sequence of elements taken from an alphabet of size $n$ such that any two distinct elements of the alphabet occur within distance $k$ of each other somewhere in the sequence. These sequences were introduced by Jaromczyk and Lonc to model a caching strategy for computing certain functions on large data sets such as medical images.  Let $f_k(n)$ be the shortest length of any $k$-radius sequence. We improve on earlier estimates for $f_k(n)$ by using tilings and logarithms.  The main result is that $f_k(n)\sim \frac{1}{k}\binom{n}{2}$ as $n\rightarrow\infty$ whenever there exists a tiling of $\bZ^{\pi(k)}$ by a certain cluster of $k$ hypercubes.  In particular this result holds for infinitely many $k$, including all $k\le 194$ and all $k$ such that $k+1$ or $2k+1$ is prime.  For certain $k$, in particular when $2k+1$ is prime, we get a sharper error term using the theory of logarithms.
\end{abstract}

\maketitle

\section{Introduction}\label{sec:introduction}

Let $k$ and $n$ be positive integers.
An $n$-ary \emph{$k$-radius sequence} is a finite sequence $a_0,a_1,\ldots ,a_{m-1}$ of elements taken from an alphabet $F=\{0,1,\ldots, n-1\}$ with the following property. For all distinct $x,y\in F$, there exist $i,j\in\{0,1,\ldots,m-1\}$ such that $a_i=x$, $a_j=y$ and $|i-j|\leq k$.  In other words, any two distinct elements of $F$ occur within a distance of $k$ of each other somewhere in the sequence. We may think of a window of length $k+1$ sliding along the sequence; the condition of being a $k$-radius sequence says that we will always be able to find a window containing both $x$ and~$y$, whatever our choice of $x$ and $y$ may be. For example, a $5$-ary $2$-radius sequence of length $7$ is:
\[
0,\,1,\,2,\,3,\,4,\,0,\,1\,.
\]
We write $f_k(n)$ for the shortest length $m$ of an $n$-ary $k$-radius sequence. The example above shows that $f_2(5)\leq 7$ (and in fact $f_2(5)=7$).

The notion of a $k$-radius sequence was introduced by Jaromczyk and Lonc~\cite{JL}. Such sequences describe a First-In First-Out caching strategy for computing functions that require computations on all pairs taken from a set of $n$ large objects (such as medical images), where at most $k+1$ objects are cached at any one time.
We believe that the problem of constructing short $k$-radius sequences is interesting as pure combinatorics; this paper establishes surprising (to us) connections to the theory of tilings and logarithms, and to properties of cyclotomic fields.

Let $k$ be fixed. This paper is concerned with establishing the asymptotic growth of $f_k(n)$. It is obvious that $f_k(n)\leq 2\binom{n}{2}$, as the concatenation of all length $2$ sequences $x,y$ with $0\leq x<y<n$ is a $k$-radius sequence. Moreover (see Section~\ref{sec:preliminaries}) it is not difficult to show that
$\frac{1}{k}\binom{n}{2}\leq f_k(n)$. So $f_k(n)$ seems to grow like $\binom{n}{2}$, and we would like to determine the limit $\lim_{n\rightarrow\infty}f_k(n)/\binom{n}{2}$ (if it exists).

Ghosh~\cite{Ghosh} proved that
\begin{equation}\label{E:Ghosh}
f_1(n)=\left\{\begin{array}{ll}
\displaystyle\binom{n}{2}+1&\text{ when $n$ is odd},\\
\\
\displaystyle\binom{n}{2}+
n/2&\text{ when $n$ is even.}
\end{array}\right.
\end{equation}
Jaromczyk and Lonc~\cite{JL} showed that
\[
f_2(n)=\frac{1}{2}\binom{n}{2}+O(n^2/\log n).
\]
Chee, Ling, Tan and Zhang~\cite{CheeLing} observed that the explicit upper bound derived by Jaromczyk and Lonc is poor for small parameter sets, and provide good constructions of $2$-radius sequences when $n$ is small. We modify the argument of Lonc and Jaromczyk to improve the error term in their estimates (see Section~\ref{sec:2_radius}), showing that
\begin{equation}\label{E:f2n}
f_2(n)=\frac{1}{2}\binom{n}{2}+O(n^{1.525}).
\end{equation}
Moreover, we show that there are infinitely many values of $n$ such that $f_2(n)= \frac{1}{2}\binom{n}{2}+O(n)$.

Turning to the case when $k>2$, Jaromczyk and Lonc~\cite{JL} show that
\[
f_k(n)\leq \frac{1}{2\lfloor (k+1)/3\rfloor}\binom{n}{2}\bigl(1+o(1)\bigr).
\]
We show that
\begin{equation}\label{E:fkn}
f_k(n)=\frac{1}{k}\binom{n}{2}+O(n^2/\log n)
\end{equation}
whenever there exists a $\bZ$-tiling of $\bZ^{\pi(k)}$ by copies of a certain fixed cluster of $k$ hypercubes (Theorem \ref{thm:k_radius}). In particular, the result holds for all $k\leq 204$ except possibly $k=195$, and the result holds whenever at least one of $k+1$ or $2k+1$ is prime. For many of these values of $k$ (including all values of $k$ with $2k+1$ prime, or with $k+1$ prime and $8|k$) we are able to improve the error term in \eqref{E:fkn}; moreover for these values of $k$ we show that there are infinitely many values $n$ such that
\begin{equation}\label{E:fkp}
f_k(n)=\frac{1}{k}\binom{n}{2}+O(n).
\end{equation}

The structure of the remainder of the paper is as follows. In \S\ref{sec:preliminaries} we establish simple lower bounds for $f_k(n)$, and recall some of the number-theoretic tools that we shall need later.  We then describe a construction method (generalising that of \cite{JL}) that for a prime number $p$ produces $p$-ary $k$-radius sequences via coverings of $\mathbb{Z}_p^*$  by arithmetic progressions.  We use this to produce a tight upper bound for $f_2(p)$ (Theorem \ref{thm:2_radius}), and deduce (\ref{E:f2n}).

In \S\ref{sec:logs_and_tilings} we recall and develop the theory of \emph{logarithms} (bijections from $\{1,\dots,k\}$ to $\mathbb{Z}_k$ such that $f(ab)=f(a)+f(b)$ whenever all three function values are defined).  We introduce the notion of a \emph{special KM-logarithm} which imposes precisely the right parity conditions that we need for a later construction of $k$-radius sequences.  We explain the relationship between logarithms and certain tilings of $\mathbb{Z}^r$ (Theorem~\ref{thm:tiling}).  We prove our main result (Theorem~\ref{thm:k_radius}) in \S\ref{sec:k_radius}, that the existence of a suitable tiling implies the estimate (\ref{E:fkn}); in particular, the existence of a logarithm of length $k$ is enough for this estimate to hold.

In \S\ref{sec:k_radius_primes} we define the notion of a $k$-radius prime, and show that if $n$ is a $k$-radius prime then we get the estimate~\eqref{E:fkp} for $f_k(n)$ with best-possible main term and a good error term (Proposition~\ref{prop:k_radius_prime_construction}).  We establish a link with logarithms by showing that $k$-radius primes exist precisely when there is a special KM-logarithm of length~$k$ (Theorem \ref{thm:stronglog_iff_k_radius}).  In \S\ref{sec:density} we derive the density of $k$-radius primes for any fixed~$k$, and use this result to improve the error term in our estimate~\eqref{E:fkn} for those values of $k$ where a special KM-logarithm of length $k$ exists. In \S\ref{sec:computations} we illustrate the applicability of our constructions by computing lists of all $k$ up to $300$ for which there exist logarithms or special KM-logarithms. Finally, in \S\ref{sec:open}, we discuss some of the various open problems that remain.

\section{Preliminaries}
\label{sec:preliminaries}

This section contains some elementary lemmas on $k$-radius sequences, and briefly describes some of the notation and results from number theory we will use.

\subsection{Elementary results on $k$-radius sequences}

\begin{lemma}
\label{lem:parameter_change}
Let $k$, $n$ and $x$ be positive integers. Then
\begin{itemize}
\item[(i)] $f_{k+1}(n)\leq f_k(n)$\emph{;}
\item[(ii)] $f_k(n)\leq f_k(n+x)$.
\end{itemize}
Indeed, the inequality \emph{(ii)} may be improved to
\[
f_k(n)\leq f_k(n+x)-\bigl\lceil \bigl(x/(n+x)\bigr)f_k(n+x)\bigr\rceil\leq \bigl(1-x/(n+x)\bigr)f_k(n+x).
\]
\end{lemma}
\begin{proof}
The first inequality of the lemma follows since every $k$-radius sequence is a $(k+1)$-radius sequence.

To establish the final inequality of the lemma, note that we may construct an $n$-ary $k$-radius sequence by taking the shortest $(n+x)$-ary $k$-radius sequence and deleting every occurrence of the $x$ most common of the $n+x$ alphabet symbols. The $x$ most common symbols together occur at least $\lceil (x/(n+x))f_k(n+x)\rceil$ times, and so the final statement of the lemma follows.
\end{proof}


\begin{lemma}[Jaromczyk and Lonc, \cite{JL}]
\label{lem:lower_bound}
For any positive integers $k$ and $n$,
\[
\frac{1}{k}\binom{n}{2}< f_k(n).
\]
\end{lemma}
\begin{proof}
Let $a_1,a_2,\ldots ,a_m$ be a $k$-radius sequence. The $k$-radius condition implies that every unordered pair of distinct alphabet symbols must occur at positions $i$ and $j$ in the sequence, where $1\leq i<j\leq m$ and $|i-j|\leq k$. The number of positions $i,j$ where $1\leq i<j\leq m$ and $|i-j|\leq k$ is strictly less than $km$, since there are $m-1$ choices for $i$ and once $i$ is fixed there at most $k$ choices for $j$. Hence $km> \binom{n}{2}$, as required.
\end{proof}
We remark that Jaromzcyk and Lonc~\cite{JL} show that the linear term of this bound can be improved, by considering the elements at the start and end of the sequence in more detail; they use this technique to give an explicit improved lower bound in the case when $k=2$.

\subsection{Some number theory}

We use various pieces of standard terminology and notation from number theory.  We use $\mathbb{Z}_n$ to denote the ring of integers modulo $n$, represented by the elements $0$, $1$, \dots, $n-1$.  The multiplicative group of invertible elements we denote $\mathbb{Z}_n^*$, its order being given by the Euler totient function $\varphi(n)$.  The letter $p$ will always indicate a prime number.  
We use $\left(\frac{a}{p}\right)$ for the Legendre symbol.  A natural number $n$ is said to be $y$-smooth if each prime $p$ dividing $n$ satisfies $p\le y$.  The function $\pi(a)$ counts the number of primes in the interval $[1,a]$, and $\omega(n)$ denotes the number of distinct prime factors of $n$.

The following theorem will be useful in moving from estimates valid for primes~$p$ to general estimates for natural numbers $n$.
\begin{theorem}[Harman, \cite{Harman}, Theorem~10.8]
\label{thm:harman}
Let $a$ and $s$ be coprime positive integers. Then for all sufficiently large $n$, there exists a prime $p$ such that $p\equiv a\bmod s$ and $n\leq p\leq n+n^{0.525}$.
\end{theorem}


\section{A construction}
\label{sec:construction}

In this section, we describe a construction method that we shall use to derive our asymptotic results. The method is a straightforward generalisation of a construction of Jaromczyk and Lonc~\cite{JL}.

Let $p$ be a prime number. For an element $d\in\bZ_p^*$, we define the set $B_{k,p}(d)\subseteq \bZ_p^*$ by
\begin{align*}
B_{k,p}(d)&=d\{\pm 1, \pm 2, \ldots, \pm k\}\\
&= \{-kd,-(k-1)d,\ldots,-2d,-d,d,2d,\ldots,(k-1)d,kd\}.
\end{align*}

For any $d\in \bZ_p^*$, let $\bfs_d$ be the periodic sequence $(0,d,2d,\ldots,(p-1)d,0,d,\dots)$ of period~$p$. Note that distinct elements $x,y\in \bZ_p$ appear at distance $k$ or less somewhere within $\bfs_d$ if and only if $y-x\in B_{k,p}(d)$.

\begin{theorem}
\label{thm:main_construction}
Let $D\subseteq \bZ_p^*$ have the property that the sets $B_{k,p}(d)$ where $d\in D$ cover $\bZ^*_p$. Then there exists a $p$-ary $k$-radius sequence of length $|D|(p+k-1)+1$.
\end{theorem}
\begin{proof}
Let $D=\{d_1,d_2,\ldots,d_u\}$, where $u=|D|$. Let $\bft_i$ be a finite sequence of length $p+k$ consisting of $p+k$ consecutive terms of the periodic sequence $\bfs_{d_i}$. For $i\geq 2$, we choose the first element of $\bft_i$ to be equal to the final element of $\bft_{i-1}$. Note that the sequence $\bft_i$ contains all pairs $(x,y)$ with $y-x\in B_{k,p}(d_i)$. The sequence $\bft_1\bft_2\cdots\bft_u$ of length $u(p+k)$ produced by concatenating the sequences $\bft_i$ is therefore a $k$-radius sequence, by the `covering' property of $D$. But the final element of each sequence $\bft_i$ is equal to the first element of the sequence $\bft_{i+1}$, by our choice of the shift of $\bfs_{d_i}$ for $i\geq 2$. We may always replace any repeats in a $k$-radius sequence by single occurrences, to produce a shorter $k$-radius sequence. There are $u-1$ places where we may do this in our construction, so we have a $k$-radius sequence of length $u(p+k-1)+1$, as required.
\end{proof}

\section{2-radius sequences}
\label{sec:2_radius}

In this section, we establish bounds on $f_2(n)$. Our methods are similar to those of Jaromczyk and Lonc~\cite{JL}, but our slightly different approach gives a better error term than theirs. This section is a good introduction to our more general methods in later sections of the paper.

\begin{theorem}
\label{thm:2_radius}
Let $p$ be an odd prime number, and let $\ell$ be the multiplicative order of $2$ modulo $p$. Then
\[
f_2(p)\leq \left\{\begin{array}{cl}
\frac{1}{4}(p+1)(p-1)(1+\tfrac{1}{\ell})+1&\text{ if $\ell$ is odd};\\
\\
\frac{1}{4}(p+1)(p-1)(1+\tfrac{2}{\ell})+1&\text{ if $\ell\equiv 2\bmod 4$};\\
\\
\frac{1}{4}(p+1)(p-1)+1&\text{ if $\ell\equiv 0\bmod 4 $}.
\end{array}\right.
\]
\end{theorem}
\begin{proof}
Suppose that $\ell$ is odd. Then the theorem follows by Theorem~\ref{thm:main_construction}, provided that we
can find a set of
$(p-1)(1+1/\ell)/4$
sets $B_{2,p}(d)$ that cover $\bZ_p^*$. Let $\langle 2\rangle$ be the multiplicative group generated by $2$ in $\bZ_p^*$. Now $\bZ_p^*$ is the disjoint union of the cosets of $\langle 2 \rangle$:
\begin{equation}
\label{eqn:2_orbits}
\bZ_p^*=C_1\cup C_2\cup\cdots \cup C_t
\end{equation}
where $C_i=\{c_i,2c_i,\ldots,2^{\ell-1}c_i\}$ for some elements $c_i\in\bZ_p^*$. Since $|C_i|=\ell$ for all $i\in\{1,2,\ldots,t\}$, we find that $t=(p-1)/\ell$. For any $i\in\{1,2,\ldots,t\}$, we have that $-C_i=C_j$ for some $j\in\{1,2,\ldots,t\}$. Since $\ell$ is odd, $-1\not\in\langle 2 \rangle$ and so $i\not=j$. So, by relabelling if necessary, we may assume that $C_{2s+1}=-C_{2s}$ for all $s\in\{1,2,\ldots ,t/2\}$. Define
\[
D=D_1\cup D_2\cup\cdots \cup D_{t/2},
\]
where
\[
D_s=\{c_{2s} 2^{2i}\mid 0\leq i< (\ell+1)/2\}
\]
for all $s\in \{1,2,\ldots,t/2\}$. Since $d,2d\in B_{2,p}(d)$ for any $d\in \bZ_p^*$, we see that the blocks $B_{2,p}(d)$ with $d\in D_s$ cover $C_{2s}$. Since $-d,-2d\in B_{2,p}(d)$, we see that the sets $B_{2,p}(d)$ with $d\in D_s$ cover $-C_{2s}=C_{2s+1}$. Hence we have a cover of $\bZ_p^*$ of size $(t/2)(\ell+1)/2$. Since $t=(p-1)/\ell$, we have a cover with the properties we require and so the theorem follows in this case.

Now suppose that $\ell$ is even. This implies that $2^{\ell/2}=-1$ in $\bZ^*_p$. Writing $\bZ^*_p$ as the union~\eqref{eqn:2_orbits} of the cosets of $\langle 2 \rangle$, we see that
\[
C_i=\{c_i2^t\mid 0\leq t< \ell/2\}\cup -\{2^tc_i\mid 0\leq t< \ell/2\}.
\]
So, defining
\[
D_i=\{2^{2j}c_i\mid 0\leq j<\lceil \ell/4\rceil\},
\]
we see that $C_i$ is covered by the $\lceil \ell/4\rceil$ sets $B_{2,p}(d)$ with $d\in D_i$. Defining
\[
D=D_1\cup D_2\cup\cdots \cup D_t
\]
we have constructed a covering of $\bZ_p^*$ consisting of $t \lceil \ell/4\rceil$ sets $B_{2,p}(d)$. Since $t=(p-1)/\ell$, the theorem now follows by Theorem~\ref{thm:main_construction}.
\end{proof}

\begin{corollary}
\label{cor:2_radius_prime}
Let $p$ be a prime such that $p\equiv 5\bmod 8$. Then
\[
f_2(p)\leq \frac{1}{2}\binom{p}{2}+
(p+3)/4\,.
\]
\end{corollary}
\begin{proof}
If $p\equiv 5\bmod 8$, then 
$2$ is a quadratic non-residue modulo $p$. The quadratic residues are exactly those elements of $\bZ_p^*$ whose orders divide $(p-1)/2$, and hence $4$ must divide the order of $2$. The corollary then follows by the final case of Theorem~\ref{thm:2_radius}.
\end{proof}

\begin{corollary}
\label{cor:2_radius}
The length $f_2(n)$ of the shortest $n$-ary $2$-radius sequence satisfies
\[
f_2(n)=\frac{1}{2}\binom{n}{2}+O(n^{1.525}).
\]
\end{corollary}
\begin{proof}
When $n$ is sufficiently large, Theorem~\ref{thm:harman} implies that there is a prime $p$ such that $p\equiv 5 \bmod 8$ and $n\leq p\leq n+n^{0.525}$. Since $f_2(n)\leq f_2(p)$ by Lemma~\ref{lem:parameter_change}, the upper bound of the corollary follows by Corollary~\ref{cor:2_radius_prime}. The lower bound follows by Lemma~\ref{lem:lower_bound}.
\end{proof}

\section{Logarithmic functions, logarithms and tilings}\label{sec:logs_and_tilings}

The results from Sections~\ref{sec:construction} and~\ref{sec:2_radius} indicate that good coverings (even better, tilings) of $\bZ_p^*$ by sets of the form $B_{k,p}(d)$ will give rise to short $k$-radius sequences. This section defines the notion of a logarithm of length $k$, and explains how to use these logarithms to construct certain high-dimensional tilings. In Section~\ref{sec:k_radius}, we will use these tilings to construct good coverings of $\bZ_p^*$, which in turn will produce short $k$-radius sequences.

\subsection{Logarithmic functions and logarithms}\label{SS:logs}
Given $k\geq1$, a \emph{logarithmic function} (of \emph{length} $k$) is a map $f:\{1,\dots,k\}\rightarrow \bZ_k$ such that $f(ab)=f(a)+f(b)$ whenever $a,b\in \{1,2,\ldots,k\}$ satisfy $ab\leq k$.  A \emph{logarithm} is a bijective logarithmic function.  Logarithms are used to construct lattice tilings of $n$-dimensional space by semi-crosses; see Stein and Szab\'o~\cite[Chapter 3]{SteinSzabo}. They also arise in group theory, number theory and coding theory; see Galovich and Stein~\cite{GalovichStein}, Forcade and Pollington~\cite{ForcadePollington}, and Gordon~\cite{Gordon}.

Galovich and Stein~\cite{GalovichStein} note that a logarithm of length $k$ exists whenever either $k+1$ is prime or $2k+1$ is prime. To see this when $k+1$ is prime, let $\alpha$ be a primitive root modulo $k+1$. Then the `discrete logarithm' map, taking $a\in\{1,2,\ldots,k\}$ to the element $x\in \bZ_k$ such that $a\equiv\alpha^x\bmod{k+1}$, is a logarithm of length~$k$. When $2k+1$ is prime, let $\alpha$ be a primitive root modulo $2k+1$. Let $g:\{1,2,\ldots,k\}\rightarrow \bZ_{2k}$ be the map taking $a\in\{1,2,\ldots,k\}$ to the element $x\in \bZ_{2k}$ such that $a\equiv\alpha^x\bmod{2k+1}$. The composition of $g$ with the natural surjection from $\bZ_{2k}$ to $\bZ_k$ (sending $x\bmod{2k}$ to $x\bmod{k}$) is a logarithm of length $k$. Computations by Galovich and Stein~\cite{GalovichStein}, and later by Forcade and Pollington~\cite{ForcadePollington}, show that logarithms are common for small values of $k$. In particular, their computations show that logarithms exist for all $k$ with $k\leq 204$, except $195$. We have verified and extended Forcade and Pollington's computations; we report our results in Section~\ref{sec:computations}.

Let $p$ be a prime such that $p\equiv 1 \bmod k$, and let $\zeta$ be a primitive $k$th root of unity in $\bZ_p^*$. Let $f$ be the map $f:\{1,2,\ldots,k\}\rightarrow \bZ_k$ defined by the equation $x^{(p-1)/k}\equiv\zeta^{f(x)}\bmod{p}$.
Then $f$ is certainly a logarithmic function of length $k$.
If $f$ happens to be bijective, then $f$ is a logarithm.
We say (following Galovich and Stein \cite{GalovichStein}) that $f$ is a Kummer--Mills-logarithm, or \emph{KM-logarithm}, if it arises in this way for some prime $p$.
(The Galovich-Stein definition for a KM-logarithm  $f$ is that $f$ extends to a $k$-character $\mathbb{Z}_p^*\rightarrow\mathbb{Z}_k$.  The relevant $k$-characters of $\mathbb{Z}_p^*$, restricted to $\{1,\dots,k\}$, correspond precisely to the choices for our function $f$ as $\zeta$ varies over the primitive $k$th roots of unity.)

\begin{theorem}[immediate from Theorem 3 of Mills~\cite{Mills}]
\label{thm:Mills}
Every logarithm of odd length $k$ is a KM-logarithm. If $k\equiv2\bmod{4}$ then a logarithm $f$ of length $k$ is a KM-logarithm if and only if $f(m)$ is even whenever both $m$ divides $k$ and $m\equiv1\bmod{4}$. If $k\equiv0\bmod{4}$, then a logarithm $f$ is a KM-logarithm if and only if $f(m)$ is even whenever $m$ divides $k/4$.  If $f$ is a KM-logarithm of length $k$, then there are infinitely many primes $p$ that exhibit $f$ as a KM-logarithm.
\end{theorem}

A logarithm $f$ of length $k$ will be called a \emph{special KM-logarithm} if either (i) $k$ is odd, or (ii) $k$ is even and $f(m)$ is even whenever $m$ divides $k/2$.  Any special KM-logarithm is a KM-logarithm, and indeed unless $k\equiv 2$, $4$, or $6 \bmod{8}$, the concepts of a KM-logarithm and a special KM-logarithm coincide. (For $k$ odd this is clear.  If $8 | k$  and $f$ is a KM-logarithm then from Theorem \ref{thm:Mills} we see that $f(m)$ is even whenever $m$ divides $k/4$.  In particular this holds for prime values of $m$ dividing $k/4$.  Since $8|k$, the primes dividing $k/4$ are the same as those dividing $k/2$.  Hence $f(m)$ is even for each prime $m$ dividing $k/2$, and hence for all $m$ dividing $k/2$.  Thus $f$ is a special KM-logarithm.)  A \emph{special KM-logarithmic function} will naturally be a logarithmic function that satisfies the parity conditions for being a special KM-logarithm, but need not be bijective.

\begin{lemma}\label{lem:stronglogs}
There are infinitely many $k$ for which there is a special KM-logarithm of length~$k$.
Indeed we have:

(i) if $p=k+1$ is prime and $8 | k$, then there is a special KM-logarithm of length~$k$;

(ii) if $p=2k+1$ is prime, then there is a special KM-logarithm of length~$k$.
\end{lemma}

\begin{proof}
(i) For such $p$ any KM-logarithm is a special KM-logarithm, and a discrete logarithm in $\bZ_p^*$ is a KM-logarithm.

(ii) Let $f$ be a logarithm of length $k$ induced by a discrete logarithm in $\bZ_p^*$.  If $k$ is odd, then we are done: $f$ is a special KM-logarithm.  If $k$ is even, then $p\equiv 1\bmod{4}$, and hence for any odd prime $q$ dividing $k$ we have $\left(\frac{q}{p}\right)=\left(\frac{p}{q}\right)=1$, and since $k$ is even we conclude that $f(q)$ must be even.  If $4\mid k$ then $p\equiv1 \bmod{8}$, so $\left(\frac{2}{p}\right)=1$, and hence $f(2)$ must be even.  We see that in all cases $f$ is a special KM-logarithm.
\end{proof}

\subsection{Logarithms and tilings}

Let $r$ be a positive integer, and let $\cC\subseteq \bZ^r$ be a finite set; we refer to $\cC$ as a \emph{cluster}.
A \emph{$\bZ$-tiling} of $\bZ^r$ by the cluster $\cC$ is a set $L\subseteq \bZ^r$ of points such that the union $\bigcup_{x\in L} (x+\cC)$ is a disjoint union and is equal to $\bZ^r$. If $L$ is an integer lattice of determinant $k$ (i.e., $L$ is a subgroup of $\bZ^r$ of index $k$), then $L$ is a $\bZ$-tiling of $\bZ^r$ by $\cC$ if and only if $\cC$ is a set of coset representatives for $L$; in other words if and only if the natural map from $\bZ^r$ to the quotient $\bZ^r/L$ is bijective when restricted to $\cC$. The tiling terminology comes from imagining $\cC$ as a set of $r$-dimensional unit hypercubes, indexed by the positions in $\bZ^r$ of their centres. A $\bZ$-tiling is then a tiling of $\bR^r$ by translates of $\cC$ in the usual sense.

We now show how to use logarithms to produce certain tilings. Let $q_1,q_2,\ldots,q_r$ be the primes in the set $\{1,2,\ldots,k\}$; so $r=\pi(k)$.  Let $\mathbb{N}_0$ be the set of non-negative integers, and let $\cC_k$ be the cluster of points in $\bZ^r$ defined by
\begin{equation}\label{E:cluster}
\cC_k=\left\{(i_1,i_2,\ldots,i_r)\in \mathbb{N}_0^r\mid \prod_{j=1}^r q_j^{i_j}\leq k\right\}.
\end{equation}
We identify the points in $\cC_k$ with the integers $1$, \dots, $k$.

\begin{theorem}
\label{thm:tiling}
Let $k$ be a fixed positive integer, let $r=\pi(k)$ and define $\cC_k$ as above. If there exists a logarithm of length $k$, then there exists a $\bZ$-tiling of $\bZ^r$ by $\cC_k$.
\end{theorem}
\begin{proof}
Let $f$ be a logarithm of length $k$. Define a group homomorphism $\phi:\bZ^r\rightarrow \bZ_k$ by
\[
\phi\bigl((i_1,i_2,\ldots,i_r)\bigr)=i_1f(q_1)+i_2f(q_2)+\cdots+i_rf(q_r),
\]
so $\phi$ sends the $j$th standard basis vector to $f(q_j)$.

Define the lattice $L\subseteq \bZ^r$ to be the kernel of $\phi$. By our remarks above, the union $\bigcup_{x\in L} (x+\cC_k)$ is a tiling if and only if $\phi$ is a bijection when restricted to $\cC_k$. Now,
\[
\phi\bigl((i_1,i_2,\ldots,i_r)\bigr)=f(q_1^{i_1}q_2^{i_2}\cdots q_r^{i_r})
\]
for all $(i_1,i_2,\ldots,i_r)\in \cC_k$, since $f(ab)=f(a)+f(b)$ whenever $a,b,ab\in\{1,2,\ldots,k\}$. So the restriction of $\phi$ to $\cC_k$ is equal to $f$. Since $f$ is a bijection, we have the tiling we require.
\end{proof}

\section{$k$-radius sequences}
\label{sec:k_radius}

This section contains the main result of our paper: the leading term of $f_k(n)$ for infinitely many values of~$k$ (including all $k\leq 194$). We need to establish a simple geometric lemma before we begin the proof.

For linearly independent vectors $x_1,x_2,\ldots,x_r\in \bR^r$, we define the $r$-dimensional parallelotope $P(x_1,x_2,\ldots,x_r)$ to be the region
\[
P(x_1,x_2,\ldots,x_r)=\left\{\sum_{i=1}^ra_ix_i\mid 0\leq a_i\leq 1\right\}.
\]

\begin{lemma}
\label{lem:parallelotope}
Let $r$ be a fixed positive integer, and let $\delta$ be a positive real number. For any linearly independent vectors $x_1,x_2,\ldots,x_r\in \bR^r$ with $\el{x_1}\leq \el{x_2}\leq\cdots\leq \el{x_r}$, let $P$ be the parallelotope $P=P(x_1,x_2,\ldots,x_r)$, and let $Q$ be the set of points in $\bR^r$ within a distance $\delta$ of $P$. Then
\[
\vol{r}{Q}\leq\vol{r}{P}+O\left(\prod_{i=2}^r\el{x_i}\right),
\]
where the implicit constant depends only on $r$ and $\delta$.
\end{lemma}
\begin{proof}
The parallelotope has a bounded number of proper subfaces $F$ (it has $\binom{r}{i}2^{r-i}$ subfaces of dimension $i$). For a proper subface $F$, define $Q_F$ to be the set of points $v\in\bR^r$ such that (i)~$v$~is within a distance $\delta$ of $F$, and (ii)~the nearest point in $F$ to $v$ does not lie in a proper subface of $F$. Then
\[
Q=P\cup\bigcup_F Q_F,
\]
where the union runs over all proper subfaces $F$ of $P$. This decomposition in the case $r=2$ is illustrated in Figure~\ref{fig:parallelogram}. Here, the parallelotope $P$ (a parallelogram as we are in dimension $2$) is in bold; the circles are the sets $Q_F$ when $F$ has dimension $0$ and the rectangles are $Q_F$ when $F$ has dimension~$1$.
\begin{figure}
\begin{center}
\includegraphics[width=100mm]{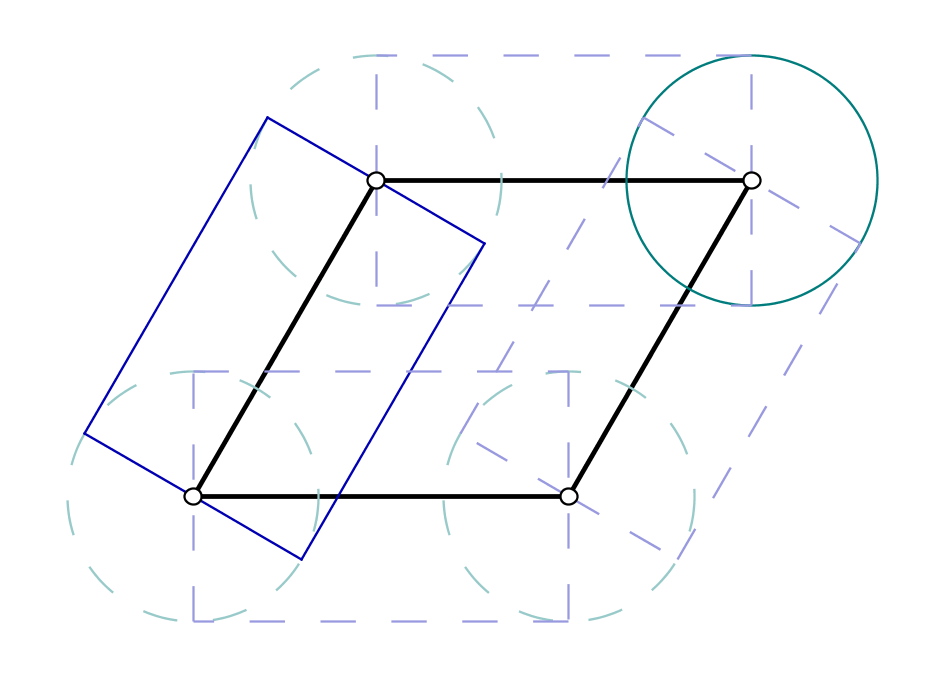}
\end{center}
\caption{Decomposing $Q$}
\label{fig:parallelogram}
\end{figure}
This decomposition shows that
\begin{equation}
\label{eqn:decomposition}
\vol{r}{Q}\leq \vol{r}{P}+\sum_F\vol{r}{Q_F}.
\end{equation}

It is not difficult to see that when $F$ is $i$-dimensional, $Q_F$ is congruent to $F\times B_{r-i}$, where $B_{r-i}$ is the $(r-i)$-dimensional ball of radius $\delta$. Now $F$ is congruent to an $i$-dimensional parallelotope, spanned by a proper subset $S\subseteq\{1,2,\ldots,r\}$ of the vectors $x_i$. So
\[
\vol{i}{F}\leq \prod_{j\in S} \el{x_j}\leq \prod_{j=r-i+1}^r \el{x_j}\leq \prod_{j=2}^r \el{x_j}.
\]
Since the ball $B_{r-i}$ is contained in a hypercube with sides of length $2\delta$, we find that $\vol{r-i}{B_{r-i}}\leq (2\delta)^r$. Hence
\[
\vol{r}{Q_F}=\vol{i}{F}\times \vol{r-i}{B_{r-i}}\leq (2\delta)^r\prod_{j=2}^r \el{x_j}=O\left(\prod_{j=2}^r \el{x_j}\right).
\]
The lemma now follows by~\eqref{eqn:decomposition}, since the number of terms in the sum in~\eqref{eqn:decomposition} is bounded.
\end{proof}

\begin{theorem}
\label{thm:k_radius}
Let $k$ be a fixed positive integer, and let $r=\pi(k)$. Suppose there exists a $\bZ$-tiling $L$ of $\bZ^r$ by $\cC_k$ \emph{(}as defined by~\eqref{E:cluster}\emph{)}. Then
\[
f_k(n)=\frac{1}{k}\binom{n}{2}+O(n^2/\log n).
\]
\emph{(}Recall that such a tiling exists whenever there is a logarithm of length $k$, so in particular when $k\leq 194$ or when either $k+1$ or $2k+1$ is prime.\emph{)}
\end{theorem}
\begin{proof}
By (\ref{E:Ghosh}) we may assume that $k>1$. Let $q_1,q_2,\ldots ,q_r$ be the primes in the set $\{1,2,\ldots,k\}$, where $q_1=2$. Let $p$ be any prime such that $\Leg{-1}{p}=-1$ and $\Leg{q_i}{p}=1$ for all $i\in\{1,2,\ldots,r\}$. A sufficient condition for a prime $p$ to satisfy these character constraints is that $p\equiv-1\bmod 8q_2q_3\cdots q_r$, so by Theorem~\ref{thm:harman} we find that for all sufficiently large $n$ there exists such a prime $p$ with $n\leq p\leq n+\lfloor n^{0.525}\rfloor$. We aim to show that
\[
f_k(p)\leq \frac{1}{k}\binom{p}{2}+O(p^2/\log p).
\]
This suffices to prove the theorem, since we may then deduce that
\begin{align*}
f_k(n)\leq f_k(p)&\leq \frac{1}{k}\binom{n+\lfloor n^{0.525}\rfloor}{2}+O\bigl((n+n^{0.525})^2/\log n\bigr)\\
&=\frac{1}{k}\binom{n}{2}+O(n^2/\log n),
\end{align*}
as required.

Recall the definition of the sets $B_{k,p}(d)$ from Section~\ref{sec:construction}. We will establish the bound we require by constructing a good covering of $\bZ^*_p$ by sets of the form $B_{k,p}(d)$ and then applying Theorem~\ref{thm:main_construction}.

Define the subgroup $H$ of the multiplicative group $\bZ_p^*$ by $H=\langle q_1,q_2,\ldots,q_r\rangle$. Define $\ell=|H|$. We may write $\bZ_p^*$ as the disjoint union of the cosets of $H$:
\[
\bZ_p^*=C_1\cup C_2\cup\cdots \cup C_t,
\]
where $C_i=c_i H$ for some elements $c_i\in\bZ_p^*$. Just as in the proof of Theorem~\ref{thm:2_radius}, we note that $|C_i|=\ell$ for all $i\in\{1,2,\ldots,t\}$ and $t=(p-1)/\ell$.

Since $\Leg{q_i}{p}=1$ for $i\in\{1,2,\ldots,r\}$, we see that $H$ consists of quadratic residues modulo $p$ and so all the elements in a coset share the same Legendre symbol. For any $i\in\{1,2,\ldots,t\}$, $-C_i=C_j$ for some $j\in\{1,2,\ldots,t\}$. Since $\Leg{-1}{p}=-1$, the Legendre symbols of the elements of $C_i$ have opposite sign to the Legendre symbols in $C_j$, and so in particular $i\not=j$. Hence $t$ is even and, by relabelling cosets if necessary, we may assume that $C_{2s+1}=-C_{2s}$ for all $s\in\{1,2,\ldots,t/2\}$.

For $d\in\bZ_p^*$, define
\[
A_{k,p}(d)=d\{1,2,\ldots,k\}=\{d,2d,\ldots,kd\}.
\]
Note that $B_{k,p}(d)=A_{k,p}(d)\cup-A_{k,p}(d)$. We aim to show that there exists an integer $w$ such that $w\le(\ell/k)\bigl(1+O(1/\log p)\bigr)$ with the following property: there exist elements $d_1,d_2,\ldots,d_w\in H$ such that the sets $A_{k,p}(d_i)$ cover $H$. This is sufficient to prove the theorem; we can see this as follows. Since $B_{k,p}(d)=A_{k,p}(d)\cup-A_{k,p}(d)$, the $w$ sets $B_{k,p}(d_i)$ cover $H\cup -H$. So the $wt/2$ sets $B_{k,p}(c_{2s}d_i)$ where $i\in\{1,2,\ldots ,k\}$ and $s\in \{1,2,\ldots,t/2\}$ cover $\bZ_p^*$. Theorem~\ref{thm:main_construction} now implies that
\begin{align*}
f_k(p)&\leq \frac{tw}{2}(p+k-1)+1\\
&\leq \frac{(p-1)w}{2\ell}(p+k)\\
&=\frac{(p-1)\bigl(1+O(1/\log p)\bigr)}{2k}(p+k)\\
&=\frac{1}{k}\binom{p}{2}+O(p^2/\log p),
\end{align*}
as required.

It remains to show the existence of the elements $d_1,d_2,\ldots,d_w\in H$ with the properties we need. Let $L\subseteq \bZ^r$ be such that $\{z+\cC_k:z\in L\}$ is a $\bZ$-tiling of $\bZ^r$. Let $\phi:\bZ^r\rightarrow H$ be the group homomorphism sending the $i$th standard basis vector of $\bZ^r$ to $q_i$ for $i\in\{1,2,\ldots,r\}$. Note that $\bZ^r$ is written additively but $H$ is written multiplicatively, so $\phi(x+y)=\phi(x)\phi(y)$ for all $x,y\in \bZ^r$. Also note that, using the definition of $\cC_k$,
\[
\phi(x+\cC_k)=A_{k,p}\bigl(\phi(x)\bigr).
\]

Let $K=\ker\phi$, so $K$ is a lattice. Since $H$ is generated by $q_1,q_2,\ldots,q_r$, we see that $\phi$ is onto and so $K$ has determinant $\ell$. Let $x_1,x_2,\ldots,x_r$ be a basis for this lattice. We may choose our basis such that
\[
\prod_{i=1}^r \el{x_i}\leq 2^{r(r-1)/4} \det(K)=2^{r(r-1)/4} \ell,
\]
by LLL-reducing this basis if necessary; see Cohen~\cite[Section~2.6]{Cohen}, for example.
Moreover, by reordering the basis if needed, we may assume that $\el{x_1}\leq \el{x_2}\leq\cdots\leq\el{x_r}$.

We claim that the shortest non-zero vector in $K$ has length at least $(\log_k p)/r$, and so in particular $\el{x_1}\geq (\log_k p)/r$. To establish our claim, suppose that the vector $(i_1,i_2,\ldots,i_r)\in K$ has length less than $(\log_k p)/r$. We need to show that this vector is zero. Our bound on the length of the vector implies that $|i_j|<(\log_k p)/r$ for all $j$, and so $\sum_{j=1}^r |i_j|<\log_k p$. Define non-negative integers $i'_j$ and $i''_j$ by
$i'_j=\max(i_j,0)$, $i''_j=-\min(i_j,0)$.  Then
\[
(i'_1,i'_2,\ldots,i'_r)-(i''_1,i''_2,\ldots,i''_r)=(i_1,i_2,\ldots,i_r)\in K
\]
and so
\begin{equation}\label{E:smallvec}
q_1^{i'_1}q_2^{i'_2}\cdots q_r^{i'_r}=\phi\bigl((i'_1,i'_2,\ldots,i'_r)\bigr)=\phi\bigl((i''_1,i''_2,\ldots,i''_r)\bigr)=q_1^{i''_1}q_2^{i''_2}\cdots q_r^{i''_r}\bmod p.
\end{equation}
But both the left hand side and right hand side of the equation above are integers bounded above by $p$. For example,
\[
q_1^{i'_1}q_2^{i'_2}\cdots q_r^{i'_r}\leq k^{\sum_{j=1}^ri'_j}\leq k^{\sum_{j=1}^r|i_j|}<k^{\log_k p}=p.
\]
So no modular reduction can have taken place in (\ref{E:smallvec}) and therefore $i'_j=i''_j$ for $j\in\{1,2,\ldots ,r\}$. This implies that $i_1=i_2=\cdots=i_r=0$, and our claim follows.

Let $R\subseteq \bZ^r$ be defined by
\[
R=\left\{y\in \bZ^r\mid y=\sum_{i=1}^r a_ix_i:0\leq a_i<1\right\}.
\]
Since $R$ is the intersection of $\bZ^r$ with a fundamental region of the lattice $K$, we see that $\phi(R)=H$ (indeed, the restriction of $\phi$ to $R$ is bijective). Define $W$ to be the set of translations $z\in L$ in our tiling such that $(z+\cC_k)\cap R\not=\emptyset$, and define $w=|W|$. We have that
\[
\bigcup_{z\in W}\bigl(\phi(z+\cC_k)\bigr)=\phi\left(\bigcup_{z\in W}(z+\cC_k)\right)\supseteq \phi(R)=H,
\]
and since $\phi(z+\cC_k)=A_{k,p}\bigl(\phi(z)\bigr)$ we have a covering of $H$ consisting of $w$ sets of the form $A_{k,p}(d)$. To prove the theorem, it suffices to show that $w\leq(\ell/k)\bigl(1+O(1/\log p)\bigr)$.

Define $\overline{\cC_k}$ to be the union of those unit hypercubes in $\bZ^r$ whose centres lie in $\cC_k$. Let $\delta$ be the diameter of $\overline{\cC_k}$; note that $\delta$ depends only on $k$. Now, $R$ is contained in the parallelotope $P=P(x_1,x_2,\ldots,x_r)$ of volume $\ell$. Define $Q\subseteq \bR^r$ to be the set of points within a distance $\delta$ of $P$. When $z\in W$, the set $z+\cC_k$ contains a point of $R$ and so $z+\overline{\cC_k}\subseteq Q$. Since $W\subseteq L$ and $L$ is a $\bZ$-tiling of $\cC_k$, the interiors of the sets $z+\overline{\cC_k}$ with $z\in W$ are disjoint, and so $\vol{r}{Q}\geq \sum_{z\in W}\vol{r}{x+\overline{\cC_k}} = wk$. However, Lemma~\ref{lem:parallelotope} implies that
\begin{align*}
\vol{r}{Q}&\leq \vol{r}{P}+O\left(\prod_{i=2}^r\el{x_i}\right)\\
&=\ell+O\left(\frac{1}{\el{x_1}}\prod_{i=1}^r\el{x_i}\right)\\
&=\ell+O(\ell/\log p),
\end{align*}
the last line following since $\prod_{i=1}^r\el{x_i}=O(\ell)$ and $\el{x_1}\geq (\log_k p)/r$. Hence
\[
w\leq \frac{1}{k}\vol{r}{Q}\leq (\ell/k)\bigl(1+O(1/\log p)\bigr),
\]
as required.
\end{proof}

\section{$k$-radius primes}
\label{sec:k_radius_primes}

We now introduce the notion of a $k$-radius prime, and show that $k$-radius primes exist if and only if there exists a special KM-logarithm of length $k$. This allows us to prove (Corollary~\ref{cor:strong_k_radius}) an upper bound on $f_k(n)$ with a better error term than that of Theorem~\ref{thm:k_radius}, provided we are willing to accept a bound for infinitely many values of $n$ rather than for all values. In \S\ref{sec:density} we will determine the density of $k$-radius primes, and this will allow us, for values of $k$ where $k$-radius primes exist, to improve the error term in Theorem~\ref{thm:k_radius} without restricting $n$; see Theorem~\ref{thm:strong_k_upper_bound}.

We say that a prime $p$ is a \emph{$k$-radius prime} if the following two conditions both hold:
\begin{subequations}\begin{align} \bullet & \text{ }p\equiv 1 \bmod 2k\,;\label{E:CC} \\
\bullet & \text{ the elements } 1^{(p-1)/k}, 2^{(p-1)/k},
\ldots , k^{(p-1)/k} \text{ in } \bZ_p^* \text{ are pairwise distinct.}\label{E:EC}\end{align}\end{subequations}

Note that $z^{(p-1)/k}$ is a $k$-th root of unity in $\bZ_p^*$ for any $z\in\bZ_p^*$. So the condition (\ref{E:EC}) is quite a strong one, as there are only $k$ such roots of unity. We remark, for later use, that the congruence condition (\ref{E:CC}) implies that $(p-1)/k$ is even.

\begin{proposition}
\label{prop:k_radius_prime_construction}
Let $p$ be a $k$-radius prime. Then there exists a $p$-ary $k$-radius sequence of length $\big((p-1)/2k\big) (p+k-1)+1$. In particular,
\[
f_k(p)=\frac{1}{k}\binom{p}{2}+O(p)
\]
when $p$ is a $k$-radius prime.
\end{proposition}
\begin{proof}
Let $\alpha$ be a primitive root modulo $p$. Define $D\subseteq \bZ_p^*$ by
\[
D=\{\alpha^{ki}:0\leq i< (p-1)/2k\}.
\]
We claim that the sets $B_{k,p}(d)$ where $d\in D$ cover $\bZ_p^*$. Since $D\subseteq \bZ_p^*$ has size $(p-1)/2k$ the proposition will then follow by Theorem~\ref{thm:main_construction}. It is sufficient to show that our sets are disjoint, for then their union has size $2k|D|=p-1$. Suppose, for a contradiction, that there exist 
$0\le i_1<i_2<(p-1)/2k$ such that $B_{k,p}(\alpha^{ki_1})\cap B_{k,p}(\alpha^{ki_2})\not=\emptyset$.
By the definition of $B_{k,p}(d)$, there exist $r_1,r_2\in\{1,2,\ldots,k\}$ such that
$
r_1\alpha^{ki_1}\equiv\pm r_2\alpha^{ki_2}\bmod{p}.
$
Raising both sides of this equation to the power $(p-1)/k$ gives the equation $r_1^{(p-1)/k}\equiv(\pm 1)^{(p-1)/k}r_2^{(p-1)/k}\bmod{p}$. Since $p$ is a $k$-radius prime, the congruence condition (\ref{E:CC}) implies that $(p-1)/k$ is even. Hence 
$r_1^{(p-1)/k}\equiv r_2^{(p-1)/k}\bmod{p}$. But now the condition (\ref{E:EC}) implies that $r_1=r_2$, and hence that $\alpha^{ki_1}\equiv\pm \alpha^{ki_2}\bmod{p}$. If $\alpha^{ki_1}\equiv\alpha^{ki_2}\bmod{p}$ then clearly $i_1=i_2$ and we have our contradiction. If $\alpha^{ki_1}\equiv-\alpha^{ki_2}\bmod{p}$ then
\[
ki_1-ki_2\equiv(p-1)/2\bmod (p-1).
\]
Again we have our desired contradiction, since
our bounds on $i_1$ and $i_2$ imply that $k(i_1-i_2)$ has absolute value less than $(p-1)/2$.
\end{proof}

The next theorem provides a connection between $k$-radius primes and special KM-logarithms. It shows us that for infinitely many values of $k$ there exist $k$-radius primes.  In particular, if $2k+1$ is prime, or if $k$ is divisible by $8$ and $k+1$ is prime, then there exist infinitely many $k$-radius primes.

\begin{theorem}\label{thm:stronglog_iff_k_radius}
If there is a special KM-logarithm of length $k$, then there are infinitely many $k$-radius primes.  Conversely, if there is a $k$-radius prime then there is a special KM-logarithm of length $k$.
\end{theorem}

\begin{proof}
Let $f:\{1,\dots,k\}\rightarrow\mathbb{Z}_k$ be a special KM-logarithm.

If $k$ is odd, then Theorem~\ref{thm:Mills} gives infinitely many primes $p\equiv 1\bmod{k}$ satisfying the condition (\ref{E:EC}), and with $k$ odd these automatically satisfy the congruence condition (\ref{E:CC}) too.

If $k$ is even, then define $\tilde{f}:\{1,\dots,2k\}\rightarrow \mathbb{Z}_{2k}$ to be any logarithmic function that agrees with $f$ modulo $k$ on the primes between $2$ and $k$.  In particular, we then have that $\tilde{f}(a)\equiv f(a) \bmod{k}$ for $1\le a \le k$.  Since $f$ is a special KM-logarithm, the values of $\tilde{f}$ provide legitimate character values mod $2k$ for Theorem~\ref{thm:Mills} to apply:  there exist infinitely many primes $p\equiv 1 \pmod{2k}$ such that $a^{(p-1)/2k}\equiv z_{2k}^{\tilde{f}(a)}\bmod{p}$ for $1\le a\le k$, where $z_{2k}$ is some primitive $2k$-th root of unity in $\mathbb{Z}_p^*$.  We certainly have the congruence condition (\ref{E:CC}), but now we must check the condition (\ref{E:EC}).  Suppose that $a^{(p-1)/k} \equiv b^{(p-1)/k} \bmod{p}$.  Then $a^{(p-1)/2k}\equiv\pm b^{(p-1)/2k}\bmod{p}$, so $z_{2k}^{\tilde{f}(a)} \equiv \pm z_{2k}^{\tilde{f}(b)}\bmod{p} $.  Since $-1\equiv z_{2k}^k\bmod{p}$, we deduce that $\tilde{f}(a)\equiv \tilde{f}(b) \bmod{k}$. But $\tilde{f}$ agrees with $f$ mod $k$, and $f$ is a logarithm, so we must have $a=b$.

Conversely, suppose that $p$ is a $k$-radius prime, and let $z_k$ be a primitive $k$-th root of unity in $\bZ_p^*$.  Define $f:\{1,\dots,k\}\rightarrow \bZ_k$ by $a^{(p-1)/k} \equiv z_k^{f(a)}\bmod{p}$.  Then $f$ is certainly a logarithm.  We now use (\ref{E:CC}) to proceed via quadratic reciprocity (as in the proof of Lemma \ref{lem:stronglogs}) to verify that $f$ is a special KM-logarithm.

\end{proof}

\begin{corollary}
\label{cor:strong_k_radius}
Let $k$ be a positive integer. Suppose there exists a special KM-logarithm of length $k$. \emph{(}Recall that this assumption holds when $2k+1$ is prime, or when $k+1$ is prime and $8|k$.\emph{)} Then
$f_k(n)=\frac{1}{k}\binom{n}{2}+O(n)$ for infinitely many values of $n$.
\end{corollary}

\section{The density of $k$-radius primes}
\label{sec:density}

We believe the notion of a $k$-radius prime is a very natural one. So it is of interest to determine the density of such primes; and this can be done, using some of the theory of cyclotomic field extensions.  For technical background material, see \cite{IR}. As an application of our result on the density of $k$-radius primes, we improve the error term in Theorem \ref{thm:k_radius} for values of $k$ such that special KM-logarithms exist.

\begin{theorem}
\label{th:density}
Let $k$ be a fixed positive integer, and let $f_\textrm{spec}(k)$ be the number of special KM-logarithms of length $k$. Let
\begin{equation}\label{E:density}
c_k =
\begin{cases}
\displaystyle\frac{1}{\varphi(2k)} \cdot \frac{f_{\mathrm{spec}}(k)}{k^{\pi(k)}} & {\rm if\ }$k${\rm\ is\ odd,}\\
\\
\displaystyle\frac{1}{\varphi(2k)} \cdot \frac{f_{\mathrm{spec}}(k)2^{\omega(k/2)}}{k^{\pi(k)}} & {\rm if\ }$k${\rm\ is\ even.}
\end{cases}
\end{equation}
There exists a positive constant $A_k$ such that
the number of $k$-radius primes less than or equal to $x$ is
\[
c_k\frac{x}{\log x} + O\left(x\exp(-A_k\sqrt{\log x})\right)\,,
\]
as $x\rightarrow\infty$, where the implied constant depends on $k$.
\end{theorem}

Heuristically the main term is clear: the fraction $1/\varphi(2k)$ gives the proportion of primes satisfying (\ref{E:CC}); the remaining fraction gives the proportion of special KM-logarithmic functions that are special KM-logarithms (when $k$ is even, the parity constraint on each prime dividing $k/2$ means that the number of special KM-logarithmic functions is only $k^{\pi(k)}/2^{\omega(k/2)}$ rather than $k^{\pi(k)}$).  We shall make use of a theorem of Elliott \cite[Theorem 1]{elliott}.  See also \cite{lenstra} for a discussion of similar density results, some of these conditional on the generalized  Riemann Hypothesis.

\begin{proof}
Suppose first that $k$ is odd.  Let $\zeta_k = \exp(2\pi i/k)$.  We work in the cyclotomic field $\mathbb{Q}(\zeta_k)$, and associate any logarithmic function $f:\{1,\dots,k\}\rightarrow\bZ_k$ with the vector $\varepsilon_f = (\zeta_k^{f(1)},\zeta_k^{f(2)},\dots,\zeta_k^{f(k)})$.

If $p\equiv 1\bmod{k}$, then $p$ splits completely in $\mathbb{Q}(\zeta_k)$ as a product of $\varphi(k)$ degree-$1$ prime ideals.  Apart from a finite number of ramified primes, all degree-$1$ prime ideals in $\mathbb{Q}(\zeta_k)$ arise in this way.  For any degree-$d$ prime ideal $\mathfrak{p}$, the $k$-th power residue symbol $\left(\frac{\beta}{\mathfrak{p}}\right)_k$ is defined for $\beta\in\bZ[\zeta_k]$ to be zero if $\mathfrak{p}\mid\beta$, and otherwise is the unique power of $\zeta_k$ that is congruent to $\beta^{(p^d-1)/k}\bmod\mathfrak{p} $.  We can associate to any prime ideal $\mathfrak{p}$ the vector  $\varepsilon_\mathfrak{p}=\left( \left(\frac{1}{\mathfrak{p}}\right)_k,\dots, \left(\frac{k}{\mathfrak{p}}\right)_k \right)$.  Let $S$ be the finite set of all vectors that arise as $\varepsilon_\mathfrak{p}$ for some degree-$1$ prime ideal $\mathfrak{p}$ dividing a rational unramified prime $p$.
For a degree-$1$ prime ideal $\mathfrak{p}$ dividing $p\equiv1\bmod{k}$, we observe that the powers of
$\zeta_k$ appearing in $\varepsilon_{\mathfrak{p}}$ correspond to the powers of some primitive $k$th root of unity modulo $p$ that appear in the sequence in (\ref{E:EC}).
The different prime ideals dividing $p$ correspond to different choices of the $k$th root of unity.
For condition (\ref{E:EC}) to hold we require the entries of $\varepsilon_\mathfrak{p}$ to be distinct.

Given any vector $\varepsilon=(\varepsilon_1,\dots,\varepsilon_k)$ whose entries are powers of $\zeta_k$, Elliott's theorem \cite[Theorem 1]{elliott} gives an asymptotic estimate for the number of prime ideals $\mathfrak{p}$ with norm less than or equal to $x$ such that $\varepsilon_\mathfrak{p}=\varepsilon$, namely
\[
\frac{N(k,k)}{k^k}\pi(x) + O\left(x \exp(-A\sqrt{\log x})\right)\,,
\]
for some positive constant $A$, where
\begin{equation}\label{E:Nkk}
N(k,k) = \sum \left(\varepsilon_1^{\nu_1}\dots\varepsilon_k^{\nu_k}\right)^{-1}\,,
\end{equation}
and where the sum in (\ref{E:Nkk}) is over all $1\le\nu_1,\dots,\nu_k\le k$ such that
$
1^{\nu_1}\cdots k^{\nu_k} = \beta^k
$
for some $\beta\in\bZ[\zeta_k]$.
Since the number of prime ideals of degree greater than $1$ and norm less than or equal to $x$ is dominated by the error term in this estimate, we have the same asymptotic formula if we restrict to degree-$1$ prime ideals.
Now we observe that if $\varepsilon=\varepsilon_\mathfrak{p}$ is in $S$, then for each term in (\ref{E:Nkk}) we have
\[
\varepsilon_1^{\nu_1}\dots\varepsilon_k^{\nu_k} =  \left(\frac{1}{\mathfrak{p}}\right)_k^{\nu_1}\dots \left(\frac{k}{\mathfrak{p}}\right)_k^{\nu_k}  = \left(\frac{\beta^k}{\mathfrak{p}}\right)_k = 1\,.
\]
Hence each $\varepsilon\in S$ gives us the same main term in the asymptotic estimate for the number of degree-$1$ prime ideals of norm less than or equal to $x$ such that $\varepsilon_\mathfrak{p}=\varepsilon$, and with the same form for the error term.  Now Mills' thoerem \cite[Theorem 3]{Mills} tells us that the $\varepsilon\in S$ are precisely those that correspond to (special KM-) logarithmic functions (as we are assuming $k$ is odd there are no parity constraints), hence amongst all degree-$1$ prime ideals the proportion that yield special KM-logarithms (corresponding to $k$-radius primes) is precisely $f_{\mathrm{spec}}(k)/k^{\pi(k)}$.

For $k$ even there are a few twists, but essentially the same argument works.  We now work in $\mathbb{Q}(\zeta_{2k})$.  For each special KM-logarithmic function $f:\{1,\dots,k\}\rightarrow \bZ_k$ we associate a set of $2^{\pi(k)}$ vectors $(\varepsilon_1,\dots,\varepsilon_k)$, where each $\varepsilon_i$ is a power of $\zeta_{2k}$ such that for primes $q$ between $1$ and $k$ we have either $\varepsilon_q = \zeta_{2k}^{f(q)}$ or $\varepsilon_q = \zeta_{2k}^{f(q)+k} = -\zeta_{2k}^{f(q)}$, extended logarithmically to all other $\varepsilon_j$.  The parity conditions for $f$ being a special KM-logarithmic function are precisely those needed for Mills' theorem to tell us that for any of these $2^{\pi(k)}$ vectors $\varepsilon$ there exist infinitely many degree-$1$ prime ideals $\mathfrak{p}\in \mathbb{Q}(\zeta_{2k})$ such that $\varepsilon=\varepsilon_\mathfrak{p}=\left( \left(\frac{1}{\mathfrak{p}}\right)_{2k},\dots, \left(\frac{k}{\mathfrak{p}}\right)_{2k} \right)$.  Bearing in mind these parity conditions, we have $2^{\pi(k)}k^{\pi(k)}/2^{\omega(k/2)}$ vectors associated to special KM-logarithmic functions, and the proportion of degree-$1$ prime ideals that correspond to special KM-logarithms is $f_\text{spec}(k)2^{\omega(k/2)}/k^{\pi(k)}$.  The rest of the argument proceeds as before, using $N(2k,k)$ rather than $N(k,k)$ (the $\nu_j$ run between $1$ and $2k$).
\end{proof}

\begin{corollary}
The proportion of all primes less than or equal to $x$ that are $k$-radius primes is $c_k$, defined by \eqref{E:density}.
\end{corollary}

To illustrate this corollary, Table \ref{tab:densities} compares the predicted densities with the proportion of $k$-radius primes observed for primes up to $10^8$ and $k\le10$.

\begin{table}
\[
\begin{array}{|c|ccccc|}\hline
k & 1 & 2 & 3 & 4 & 5 \\ \hline
\text{predicted} & 1.00 & 0.250 & 0.111 & 0.00 & 0.00160 \\
\text{observed} & 1.00 & 0.250 & 0.111 & 0.00 & 0.00161 \\ \hline \hline
k & 6 & 7 & 8 & 9 & 10 \\ \hline
\text{predicted} & 0.00463 & 0.00250 & 0.000977 & 0.000610 & 0.000200 \\
\text{observed} & 0.00464 & 0.00250 & 0.000974 & 0.000600 & 0.000202 \\ \hline
\end{array}
\]
\caption{Predicted and observed (primes up to $10^8$) densities of $k$-radius primes}
\label{tab:densities}
\end{table}

\begin{corollary}\label{C:gaps}
If $f_{spec}(k)>0$, then with $A_k$ as in Theorem \ref{th:density} there is a $k$-radius prime between $n$ and $n+O\bigl(n\log n\exp(-A_k\sqrt{\log n})\bigr)$.
\end{corollary}

\begin{proof}
From Theorem \ref{th:density} we have that the number of $k$-radius primes between $n$ and $n+g$ (with $1\le g\le n$) is
\[
c_k\left(\frac{n+g}{\log(n+g)} - \frac{n}{\log n}\right) + O\bigl(n\exp(-A_k\sqrt{\log n})\bigr)\,.
\]
This is strictly positive provided that the main term $c_k\left(\frac{n+g}{\log(n+g)} - \frac{n}{\log n}\right)$ is strictly larger than the absolute value of the error term.  This main term is asymptotic to $c_k g/\log n$ as $n\rightarrow\infty$, so with $g$ larger than some constant multiple of $n\log n\exp(-A_k\sqrt{\log n})$ we are assured of a $k$-radius prime between $n$ and $n+g$.
\end{proof}

\begin{theorem}
\label{thm:strong_k_upper_bound}
Let $k$ be a fixed positive integer such that a special KM-logarithm exists. \emph{(}Recall that a special KM-logarithm exists, in particular, when $2k+1$ is prime.\emph{)} Define the positive constant $A_k$ as in Theorem~\ref{th:density}. Then
\[
f_k(n)=\frac{1}{k}\binom{n}{2}+O\big(n^2\log n \exp(-A_k\sqrt{\log n})\big).
\]
\end{theorem}
We remark that this improves the bound of Theorem~\ref{thm:k_radius} for those values of~$k$ where a special KM-logarithm exists, since $\exp(A_k\sqrt{\log n})$ grows faster than $(\log n)^2$.
\begin{proof}
Let $p$ be the smallest $k$-radius prime such that $n\leq p$. Then $p=n+g$ where $g=O\bigl(n\log n\exp(-A_k\sqrt{\log n})\bigr)$, by Corollary~\ref{C:gaps}. Proposition~\ref{prop:k_radius_prime_construction} implies that $f_k(p)=\frac{1}{k}\binom{p}{2}+O(p)$ and so
\[
f_k(n)\leq f_k(p)=\frac{1}{k}\binom{p}{2}+O(p)=\frac{1}{k}\binom{n}{2}+O(ng),
\]
as required.
\end{proof}

\section{Computational results}\label{sec:computations}

In this section, we present our computational results. We begin by describing the techniques we used to design an efficient search for logarithms, and to enumerate them. We then describe the results themselves.

\subsection{Finding logarithms}

Let $q_1<\cdots<q_r$ be the primes in $\{1,\dots,k\}$.  A logarithmic function $f$ of length $k$ is uniquely determined by the values of $f(q_1)$, \dots, $f(q_r)$, which can be chosen independently, so there are $k^{\pi(k)} = k^r$ of them.

Let $V = {\bZ_k}^k$ be the set of all vectors of length $k$, with components in $\bZ_k$.  We can conveniently represent any logarithmic function $f$ of length $k$ as the element $\mathbf{v}_f=\bigl(f(1),f(2),\ldots,f(k)\bigr)\in V$.  For each prime $q_i\leq k$, define ${\bf e}_i\in V$ by letting the $j$th component of ${\bf e}_i$ be the exponent of the largest power of $q_i$ dividing $j$. Then $f$ is a logarithmic function with $f(q_i)=a_i$ for $i\in\{1,2,\ldots,r\}$ if and only if $\mathbf{v}_f=\sum_{i=1}^r a_i\mathbf{e}_i$. So the set of logarithmic functions is a $\bZ_k$-submodule of $V$, with generating set $\mathbf{e}_1,\mathbf{e}_2,\ldots,\mathbf{e}_r$. In what follows, we identify a logarithmic function $f$ with its representation $\mathbf{v}_f$.

Logarithms are bijective logarithmic functions, so a naive strategy would be to search over all $k^r$ logarithmic functions. Forcade and Pollington~\cite{ForcadePollington} observed that it is sufficient to search over all possibilities for $f(q_i)$ with $q_i\leq k/2$. Once these values are fixed, $f$ is determined on all $(k/2)$-smooth numbers in $\{1,2,\ldots,k\}$. If the resulting partial function is injective with image $X$, then $f$ may always be completed to a logarithm by assigning the remaining values of $f(q_i)$ for $q_i>k/2$ to be an arbitrary permutation of $\bZ_k\setminus X$. This trick reduces the size of our search from $k^{\pi(k)}$ to $k^{\pi(k/2)}$. Another way of thinking of this is as follows. We partition the set of logarithmic functions into parts; two functions are in the same part if and only if they agree on $q_i$ for $q_i\leq k/2$. We search over parts rather than logarithmic functions. Each part contains $\bigl(\pi(k)-\pi(k/2)\bigr)!$ logarithms if the appropriate partial function is injective, otherwise the part contains no logarithms. In our computations, we made use of a generalisation of this, which we now describe.

Partition the primes $q_1$, \dots, $q_r$ into \emph{blocks} $B_1$, \dots, $B_t$ in following way.  Any prime $\leq\sqrt{k}$ appears on its own in a block of size $1$.  Primes $q$ in the interval $(\sqrt{k},k]$ are grouped together in blocks according to the value of $\ell$ such that $k/(\ell+1)<q\leq k/\ell$.  Observe that if $\pi$ is a permutation of $\{1,\dots,r\}$ that induces a permutation of each block, and $f = \sum_{i=1}^r a_i {\bf e}_i$ is a logarithm, then so is $\sum_{i=1}^r a_{\pi(i)} {\bf e}_i$.  Logarithms that are related in this way we call \emph{combinatorially equivalent}.

If $f$ is a logarithm of length $k$, then so is $af$ for any $a\in \bZ_k^*$.  Two logarithms $f$ and $g$ will be called \emph{linearly equivalent} if $f=ag$ for some $a\in \bZ_k^*$.
Two logarithms $f$ and $g$ will be deemed \emph{equivalent} if there is a logarithm $h$ such that $f$ is combinatorially equivalent to $h$ and $h$ is linearly equivalent to $g$.  For finding and counting logarithms it is extremely useful to work up to equivalence.

We remark that a small modification of this technique works when counting KM-logarithms or special KM-logarithms. If $f$ is a KM-logarithm (or a special KM-logarithm), then so is $af$ for any $a\in \bZ_k^*$, for if $k$ is odd then there are no parity conditions to check, whilst if $k$ is even then $a$ must be odd, and the parity of $af(j)$ (reduced modulo even $k$) is the same as the parity of $f(j)$. When dealing with special KM-logarithms or KM-logarithms, any primes dividing $k$ must be put in singleton blocks.  Since for simple logarithms the only primes $q$ in blocks of size greater than $1$ satisfy $q>\sqrt{k}$, there is at most one block that is affected by this change, having one of its primes removed to a singleton block.

Given $k$, our search for a logarithm $f$ proceeds by attempting to assign values for $f(q_1)$, \dots, $f(q_r)$ in that order.  Having found compatible values for $f(q_1)$, \dots, $f(q_{j-1})$ and having computed ${\bf v}_{j-1} = \sum_{i\leq j-1} f(q_i){\bf e}_i$, we consider each possibility for $f(q_j)$ in turn: the vector ${\bf v}_j = {\bf v}_{j-1} + f(q_j){\bf e}_j$ is computed; the $q_j$-smooth components of ${\bf v}_j$ are extracted (using a precomputed list of $q_j$-smooth numbers), sorted, and checked for repeats; if a repeat is found then the next value of $f(q_j)$ is tried, backtracking once all possibilities are exhausted; if no repeat is found then $j$ is increased.

The backtracking makes use of equivalence.  Within each block $\{q_{i_1},\dots,q_{i_t}\}$ of primes, we constrain the values of the logarithms to satisfy $f(q_{i_1})<\cdots<f(q_{i_t})$.  Linear equivalence is exploited by insisting that $f(2)\mid k$, and restricting values of $f(3)$ so that there does not exist $a$ with $\gcd(a,k)=1$ and $af(2)\equiv f(2) \bmod{k}$ such that $af(3)\bmod{k}$ is smaller than $f(3)$.

For computing KM-logarithms and special KM-logarithms, the parity conditions are checked on $q_j$-smooth numbers before accepting the legitimacy of $f(q_j)$.

\subsection{Counting logarithms}

To count the total number of logarithms (or KM-logarithms, or special KM-logarithms) of length $k$, we need to know the size of each equivalence class, and how many representatives of each class appear in the output of our algorithm (now not terminating when a logarithm is found, but continuing to search the entire space of logarithmic functions).  If $f$ is a logarithm, scaled so that $f(2)\mid k$, then the number of $a\in\bZ_k^*$ such that $af(2)\equiv f(2)\bmod{k}$ equals the number of $a\in\bZ_k^*$ with $a\equiv 1 \bmod{(k/f(2))}$, which is the size of the kernel of the natural homomorphism $\bZ_k^*\rightarrow \bZ_{k/f(2)}^*$, namely $\varphi(k)/\varphi\bigl(k/f(2)\bigr)$.  To ease the counting, we remove the constraint on $f(3)$, and check for each logarithm that the vector of logarithm values for the primes in singleton blocks scales to precisely $\varphi(k)/\varphi\bigl(k/f(2)\bigr)$ distinct vectors on multiplication by elements of $\bZ_k^*$.  If the primes below $k$ fall into blocks of lengths $\ell_1$, \dots, $\ell_t$, then each logarithm $f$ produced by our algorithm contributes
$
\varphi\bigl(k/f(2)\bigr)\prod_{j=1}^t \ell_j !
$
to the total number of logarithms.

\subsection{Results and comments}

We write $f_{\mathrm{log}}(k)$ for the number of logarithms of length $k$, and $f_{\mathrm{spec}}(k)$ for the number of special KM-logarithms. Table~\ref{tab:log_arr_count} shows $f_{\mathrm{log}}(k)$  and $f_{\mathrm{spec}}(k)$ for small values of $k$.
\begin{table}
\begin{small}
\[
\begin{array}{|c|ccccccc|}\hline
n&1&2&3&4&5&6&7\\\hline
f_{\mathrm{log}}(k)&1&1&2&2&8&10&36\\
f_{\mathrm{spec}}(k)&1&1&2&0&8&2&36\\\hline\hline
n&8&9&10&11&12&13&14\\\hline
f_{\mathrm{log}}(k)&40&24&20&140&136&936&624\\
f_{\mathrm{spec}}(k)&16&24&8&140&0&936&312\\\hline\hline
n&15&16&17&18&19&20&21\\\hline
f_{\mathrm{log}}(k)&416&256&3\,648&2\,088&30\,240&16\,704&9\,792\\
f_{\mathrm{spec}}(k)&416&96&3\,648&576&30\,240&4\,608&9\,792\\\hline\hline
n&22&23&24&25&26&27&28\\\hline
f_{\mathrm{log}}(k)&9\,000&103\,488&86\,400&72\,960&36\,576&22\,896&12\,096\\
f_{\mathrm{spec}}(k)&3\,360&103\,488&10\,368&72\,960&13\,752&22\,896&1,296\\\hline\hline
n&29&30&31&32&33&34&35\\\hline
f_{\mathrm{log}}(k)&134\,400&105\,216&2\,671\,200&1\,708\,800&794\,400&396\,288&145\,152\\
f_{\mathrm{spec}}(k)&134\,400&23\,424&2\,671\,200&556\,800&794\,400&202\,752&145\,152\\\hline\hline
n&36&37&38&39&40&41&42\\\hline
f_{\mathrm{log}}(k)&109\,440&3\,594\,240&2\,244\,672&1\,202\,688&660\,480&17\,606\,400&16\,104\,960\\
f_{\mathrm{spec}}(k)&7\,488&3\,594\,240&1\,013\,472&1\,202\,688&102\,912&17\,606\,400&2\,021\,760\\\hline
\end{array}
\]
\caption{The number of logarithms and special KM-logarithms of length $k$}
\label{tab:log_arr_count}
\end{small}
\end{table}

Table~\ref{tab:log_arr_count} shows that logarithms of length $k$ are common when $k$ is small. This entire table took just a few minutes to compute.  However, logarithms of length $k$ do not always exist: Forcade and Pollington~\cite{ForcadePollington} compute that there are no logarithms of length $195$, but that logarithms of length $k$ exist
for all other values of $k$ with $k\leq 204$. Their computations also show that there exists no logarithm of length $255$.

We pushed the computations for finding logarithms, KM-logarithms, and special KM-logarithms up to $k=300$. There is no logarithm of length $k$ for $k=195$, $205$, $208$, $211$, $212$, $214$, $217$, $218$, $220$, $227$, $229$, $235$, $242$, $244$, $246$, $247$, $248$, $252$, $253$, $255$, $257$, $258$, $259$, $263$, $264$, $265$, $266$, $267$, $269$, $271$, $274$, $275$, $279$, $283$, $286$, $287$, $289$, $290$, $291$, $294$, $295$, $297$, $298$.  For all other $k\le 300$ there exists a logarithm of length $k$.  In addition to the above list, there is no KM-logarithm of length $k$ for $k=184$, $234$, $236$.  In addition to all of these, there is no special KM-logarithm for $k=4$, $12$, $60$, $180$, $182$, $190$, $196$, $222$, $238$, $268$, $276$, $282$, $292$.  When no logarithm existed, the computations for larger values of $k$ typically took about $50$ hours each, using PARI/GP \cite{pari} on a laptop, with the worst case by far being $k=294$ (which took about $150$ hours).

Evidence from the enumeration in Table~\ref{tab:log_arr_count}, together with heuristic combinatorial considerations, seem to point towards logarithms of length $k$ being common when $k$ is prime. However, this is not the case in general:
the smallest prime $k$ for which there is no logarithm is $k=211$.

The largest known $k$ for which there exists a logarithm without either $k+1$ or $2k+1$ being prime is $k=277$. For $k$ above $200$, these `sporadic' logarithms appear to be increasingly rare: $k=201$, $202$, $203$, $206$, $207$, $213$, $223$, $225$, $234$, $236$, $237$, $241$, $272$, $277$. (Could it possibly be the case that the only large values of $k$ such that logarithms exist are when $k+1$ or $2k+1$ is prime?) To illustrate our output, we present the values of a logarithm of length $277$: it is enough to give the values at primes below $k/2$ (which determine the values at all $(k/2)$-smooth numbers up to $k$), as the values at remaining primes below $k$ may be assigned the missing logarithm values in any order.  The bold numbers in Table \ref{tab:log278} indicate the start of each of the blocks of primes; within each block, the logarithm values are seen to increase.

\begin{table}
\[
\begin{array}{|c||c|c|c|c|c|c|c|c|c|c|c|c|}\hline
p & \bf 2 & \bf 3 & \bf 5 & \bf 7 & \bf 11 & \bf 13 & \bf 17 & \bf 19 & \bf 23 & \bf 29  & \bf 31 & \bf 37 \\ \hline
f(p) & 1 & 122 & 29 & 100 & 48 & 108 & 113 & 239 & 190 & 166 & 176 & 64 \\  \hline \multicolumn{13}{c} { } \\ \hline
p &  \bf 41 & 43 & \bf 47 &  53 &\bf 59 &  61 & 67 & \bf 71 & 73 & 79 &  83 & 89 \\ \hline
f(p) & 18 & 134 & 145 & 226 & 38 & 72 & 81 & 43 & 182 & 196 & 249 & 253  \\  \hline \multicolumn{13}{c} { } \\ \cline{1-10}
p & \bf 97 & 101 & 103 & 107 & 109 & 113  & 127 & 131 & 137  \\ \cline{1-10}
f(p) & 23 & 54 & 120 & 184 & 217 & 233 & 260 & 262 & 271  \\ \cline{1-10}
\end{array}
\]
\caption{A logarithm of length $277$}
\label{tab:log278}
\end{table}

Forcade and Pollington remarked in \cite{ForcadePollington} that whenever they had found a logarithm $f$ of length $k$, they had managed to find one with $f(2)=1$.  In fact there are many exceptions to this.  For $k<288$,  the exceptional values are $k=176$, $198$, $200$, $215$, $216$, $219$, $224$, $225$, $232$, $234$, $240$, $249$, $250$, $256$, $260$, $262$, $276$, $280$, $282$: for these values of $k$ there exist logarithms of length $k$, but none with $f(2)=1$.

\section{Open problems}
\label{sec:open}

The big open question\footnote{
Recent work by the first author has used probabilistic methods to prove that this limit exists and equals $1/k$.  However the problem of explicitly constructing asymptotically good $k$-radius sequences for all $k$ remains open.
} for the asymptotics of $k$-radius sequences is the following:

\begin{open}
For any fixed $k$, determine $\lim_{n\rightarrow\infty} f_k(n)/\binom{n}{2}$ (if it exists).
\end{open}
The results in this paper show that $1/k\leq f_k(n)/\binom{n}{2}\leq 1/\bigl(k+o(k)\bigr)$ for sufficiently large $n$, but is it the case that $\lim_{n\rightarrow\infty} f_k(n)/\binom{n}{2}= 1/k$?
In particular, is it the case that $\lim_{n\rightarrow\infty} f_{195}(n)/\binom{n}{2} = 1/195$?
Our results say that this limit (if it exists) lies in the interval $[1/195,1/194]$. It is not difficult to see that whenever there exists a covering of $\bZ^{\pi(195)}$ by translated copies of ${\mathcal{C}}_{195}$ with density $1+\epsilon$, then the limit is at most $(1+\epsilon)/195$. So one approach to determining this limit might be to find good coverings of $\bZ^{\pi(195)}$ by copies of ${\mathcal{C}}_{195}$.

\begin{open}
Determine the set of integers $k$ such that there exists a logarithm of length $k$.
\end{open}
Is it the case that for all sufficiently large $k$, logarithms exist if and only if either $k+1$ or $2k+1$ is prime?
If this is not true, is it the case that the set of integers $k$ such that there exists a logarithm of length $k$ has density zero? Our computer searches for logarithms suggest this is very possible. In particular, is there an infinite sequence of integers $k$ such that no logarithms of length $k$ exist? We believe this is very likely, and give a heuristic argument below.

Let $M_k$ be the largest size of the image of a logarithmic function from $\{1,\dots,k\}$ to $\bZ_k$, and let $R_k\le k$ be maximal such that there exists a logarithmic function from $\{1,\dots,k\}$ to $\bZ_k$ whose restriction to $R_k$-smooth elements is injective.  Thus a logarithm exists if and and only if $M_k=R_k=k$, and when a logarithm does not exist the numbers $M_k$ and $R_k$ give natural measures of failure.

\begin{open} Investigate $M_k$ and $R_k$.
\end{open}

We have that $\limsup_{k\rightarrow\infty}M_k/k = \limsup_{k\rightarrow\infty}R_k/k=1$.
Is it the case that $\liminf_{k\rightarrow\infty}M_k/k = \liminf_{k\rightarrow\infty}R_k/k=0$?
We give a heuristic argument for $R_k$ (with thanks to Carl Pomerance for suggesting this line of reasoning).
Fix small $\epsilon>0$.
Let $\psi(x,y)$ denote the number of $y$-smooth numbers below $x$.
With $u = \log x/\log y$
we shall use repeatedly the estimate $\psi(x,y) = x u^{-u+o(u)}$, uniformly as $u\rightarrow\infty$ with $u<(1-\epsilon)\log x / \log\log x$ \cite{CanfieldErdosPomerance}.
Let the $r$th prime be $q_r$, and for a logarithm $f$ of length $k$, let $f_r$ be the restriction of $f$ to the $q_r$-smooth numbers between $1$ and $k$.
Suppose we have a possibility for $f_r$, with $q_r\approx \exp(\sqrt{\log k})$.
We have therefore used up $\psi(k,q_r)$ of the elements of $\bZ_k$ as values of $f$ on $q_r$-smooth numbers.
If we assume suitably random behaviour, the expected number of possibilities for $f_{r+1}$ that extend $f_r$ is bounded above by
\[
k\bigl(1-\psi(k,q_r)/k\bigr)^{\psi(k/q_{r+1},q_{r+1})}\,,
\]
since there are $\psi(k/q_{r+1},q_{r+1})$ numbers between $1$ and $k$ that are $q_{r+1}$-smooth but not $q_r$-smooth.
Now with $u=\log k/\log q_r \sim \sqrt{\log k}$ we have $\psi(k,q_r)= ku^{-u+o(u)}\ge k^{1-\epsilon}$ for all large enough $k$; and with $v = \log(k/q_{r+1})/\sqrt{\log{k}}\sim\sqrt{\log{k}}$ we have $\psi(k/q_{r+1},q_{r+1}) = kv^{-v+o(v)}/q_{r+1}\ge k^{1-\epsilon}$ for all large enough $k$.
Thus a heuristic upper bound for the expected number of possibilities for $f_{r+1}$ (given $f_r$) is  $k(1-k^{-\epsilon})^{k^{1-\epsilon}}$.
Since there are $O(k^r) = O(k^{\exp(\sqrt{\log k})})$ possibilities for $f_r$,
our heuristic bound for the number of possibilities for $f$, summed over all $k\ge k_0$, tends to $0$ as $k_0\rightarrow\infty$.
Without some structural reason (such as $k+1$ being prime), this suggests that only finitely many $k$ would achieve $R_k$ as large as $\exp(\sqrt{\log k})$ (and indeed this bound could heuristically be reduced still further).

Computations show that $M_k=k$ for $k\le 194$, and we computed further that $M_k\ge k-1$ for $k\le 257$, but that $M_{258}= 256$.

\section*{Acknowledgments}
Our interest in this topic was kindled by an inspirational talk by Yeow Meng Chee at Royal Holloway, University of London, in March 2010.  We are grateful to Glyn Harman for providing us with just the result that we needed on gaps between primes in arithmetic progressions, and thank Glyn Harman, Carl Pomerance and Igor Shparlinski for other helpful comments.
Finally we thank the referee for a careful reading of the paper and for several helpful suggestions.

\end{document}